\documentclass{amsart}%
\usepackage{amssymb}
\usepackage{amsmath}
\usepackage{amsfonts}
\usepackage{graphicx}%
\providecommand{\U}[1]{\protect\rule{.1in}{.1in}}
\newtheorem{theorem}{Theorem}
\theoremstyle{plain}

\newtheorem{corollary}{Corollary}

\newtheorem{definition}{Definition}

\newtheorem{lemma}{Lemma}

\newtheorem{proposition}{Proposition}
\newtheorem{remark}{Remark}

\numberwithin{equation}{section}
\begin{document}
\title[Young's Inequality]{An Extension of Young's Inequality}
\author{Flavia-Corina Mitroi}
\address{University of Craiova, Department of Mathematics, Street A. I. Cuza 13,
Craiova, RO-200585, Romania}
\email{fcmitroi@yahoo.com}
\author{Constantin P. Niculescu}
\address{University of Craiova, Department of Mathematics, Street A. I. Cuza 13,
Craiova, RO-200585, Romania}
\email{cpniculescu@gmail.com}
\thanks{Corresponding author: Constantin P. Niculescu}
\date{June 2011}

\begin{abstract}
Young's inequality is extended to the context of absolutely continuous
measures. Several applications are included.

\end{abstract}
\maketitle

\section{Introduction}

Young's inequality \cite{Y1912} asserts that every strictly increasing
continuous function $f:\left[  0,\infty\right)  \longrightarrow\left[
0,\infty\right)  $ with $f\left(  0\right)  =0$ and $\underset{x\rightarrow
\infty}{\lim}f\left(  x\right)  =\infty$ verifies an inequality of the
following form,%
\begin{equation}
ab\leq\int_{0}^{a}f\left(  x\right)  dx+\int_{0}^{b}f^{-1}\left(  y\right)
dy, \label{youngineq}%
\end{equation}
whenever $a$ and $b\ $are nonnegative real numbers. The equality occurs if and
only if $f\left(  a\right)  =b$. See \cite{HLP}, \cite{Mit}, \cite{NP2006} and
\cite{RV} for details and significant applications.

Several questions arise naturally in connection with this classical result.

\begin{enumerate}
\item[(Q1):] Is the restriction on strict monotonicity (or on continuity)
really necessary?

\item[(Q2):] Is there any weighted analogue of Young's inequality?

\item[(Q3):] Can Young's inequality be improved?
\end{enumerate}

F. Cunningham Jr. and N. Grossman \cite{CG1971} noticed that the question (Q1)
has a positive answer (correcting the prevalent belief that Young's inequality
is the business of strictly increasing continuous functions). The aim of the
present paper is to extend the entire discussion to the framework of locally
absolutely continuous measures and to prove several improvements.

As well known, Young's inequality is an illustration of the Legendre duality.
Precisely, the functions%
\[
F(a)=\int_{0}^{a}f\left(  x\right)  dx\text{ and }G(b)=\int_{0}^{b}%
f^{-1}\left(  x\right)  dx,
\]
are both continuous and convex on $\left[  0,\infty\right)  $ and
(\ref{youngineq}) can be restated as%
\begin{equation}
ab\leq F(a)+G(b)\text{\quad for all }a,b\in\left[  0,\infty\right)  ,
\label{youngineq2}%
\end{equation}
with equality if and only if $f\left(  a\right)  =b.$ Because of the equality
case, the formula (\ref{youngineq2}) leads to the following connection between
the functions $F$ and $G:$%
\begin{equation}
F(a)=\sup\left\{  ab-G(b):b\geq0\right\}  \label{defF}%
\end{equation}
and
\[
G(b)=\sup\left\{  ab-F(a):a\geq0\right\}  .
\]

It turns out that each of these formulas produces a convex function (possibly
on a different interval). Some details are in order.

By definition, the \emph{conjugate} of a convex function $F$ defined on a
nondegenerate interval $I$ is the function%
\[
F^{\ast}:I^{\ast}\rightarrow\mathbb{R},\text{\quad}F^{\ast}(y)=\sup\left\{
xy-F(x):x\in I\right\}  ,
\]
with domain $I^{\ast}=\left\{  y\in\mathbb{R}:F^{\ast}(y)<\infty\right\}  $.
Necessarily $I^{\ast}$ is an non-empty interval and $F^{\ast}$ is a convex
function whose level sets $\left\{  y:F^{\ast}(y)\leq\lambda\right\}  $ are
closed subsets of $\mathbb{R}$ for each $\lambda\in\mathbb{R}$ (usually such
functions are called \emph{closed} convex functions).

A convex function may not be differentiable, but it admits a good substitute
for differentiability.

The \emph{subdifferential\ }of a real function\emph{\ }$F$ defined on an
interval $I$ is a multivalued function $\partial F:I\rightarrow\mathcal{P}%
(\mathbb{R})$ defined by%
\[
\partial F(x)=\left\{  \lambda\in\mathbb{R}:F(y)\geq F(x)+\lambda(y-x)\text{,
for every}\,\,y\in I\right\}  .
\]

Geometrically, the subdifferential gives us the slopes of the supporting lines
for the graph of $F$. The sub\allowbreak differential at a point\emph{\ }is
always a convex set, possibly empty, but the convex functions $F:I\rightarrow
\mathbb{R}$ have the remarkable property that $\partial F(x)\neq\emptyset$ at
all interior points. It is worth noticing that $\partial F(x)=\left\{
F^{\prime}(x)\right\}  $ at each point where $F$ is differentiable (so this
formula works for all points of $I$ except for a countable subset). See
\cite{NP2006}, page 30.

\begin{lemma}
\label{Lem1}\noindent\emph{(}Legendre duality, \emph{\cite{NP2006}}, page
\emph{41)}. Let $F:I\rightarrow\mathbb{R}$\ be a closed convex function. Then
its conjugate $F^{\ast}:I^{\ast}\rightarrow\mathbb{R}$ is also convex and
closed and:

$i)$ $xy\leq F(x)+F^{\ast}(y)$ for all $x\in I,$ $y\in I^{\ast};$

$ii)$ $xy=F(x)+F^{\ast}(y)$ if, and only if, $y\in\partial F(x);$

$iii)$ $\partial F^{\ast}=\,\left(  \partial F\right)  ^{-1}$ \emph{(}as
graphs\emph{)}$;$

$iv)$ $F^{\ast\ast}=F.$
\end{lemma}

Recall that the inverse of a graph $\Gamma$ is the set $\Gamma^{-1}=\left\{
\left(  y,x\right)  :(x,y)\in\Gamma\right\}  .$

How far is Young's inequality from the Legendre duality? Surprisingly, they
are pretty closed in the sense that in most cases the Legendre duality can be
converted into a Young like inequality. Indeed, every continuous convex
function admits an integral representation.

\begin{lemma}
\label{Lem2}\noindent\emph{(}See \emph{\cite{NP2006}}, page \emph{37)}. Let
$F$\ be a continuous convex function defined on an interval $I$ and let
$\varphi:I\rightarrow\mathbb{R}$ be a function such that $\varphi
(x)\in\partial F(x)$ for every $x\in\,I.$\ Then for every $a<b$ in $I$ we
have
\[
F(b)-F(a)=\int_{a}^{b}\,\varphi(t)\,dt.
\]

\end{lemma}

As a consequence, the heuristic meaning of the formula $i)$ in Lemma
\ref{Lem1} is the following Young like inequality,
\[
ab\leq\int_{a_{0}}^{a}\varphi\left(  x\right)  dx+\int_{b_{0}}^{b}\psi\left(
y\right)  dy\text{\quad for all }a\in I,\ b\in I^{\ast},
\]
where $\varphi$ and $\psi$ are selection functions for $\partial F$ and
respectively $\left(  \partial F\right)  ^{-1}$. Now it becomes clear that
Young's inequality should work outside strict monotonicity (as well as outside
continuity). The details are presented in Section 2. Our approach (based on
the geometric meaning of integrals as areas) allows us to extend the framework
of integrability to all positive measures $\rho$ which are locally absolutely
continuous with respect to the planar Lebesgue measure $dxdy$. See Theorem
\ref{ThmYoungNondecr}\ below.

A special case of Young's inequality is%
\[
xy\leq\frac{x^{p}}{p}+\frac{y^{q}}{q},
\]
which works for all $x,y\geq0$, and $p,q>1$ with $1/p+1/q=1$. Theorem
\ref{ThmYoungNondecr} yields the following companion to this inequality in the
case of Gaussian measure $\frac{4}{2\pi}e^{-x^{2}-y^{2}}dxdy$ on
$[0,\infty)\times\lbrack0,\infty):$
\[
\operatorname{erf}(x)\operatorname{erf}(y)\leq\frac{2}{\sqrt{\pi}}\int_{0}%
^{x}\operatorname{erf}\left(  s^{p-1}\right)  e^{-s^{2}}ds+\frac{2}{\sqrt{\pi
}}\int_{0}^{y}\operatorname{erf}\left(  t^{q-1}\right)  e^{-t^{2}}dt,
\]
where%
\begin{equation}
\operatorname{erf}(x)=\frac{2}{\sqrt{\pi}}\int_{0}^{x}e^{-s^{2}}ds \label{erf}%
\end{equation}
is the \emph{Gauss error function} (or the erf function).

The precision of our generalization of Young's inequality makes the objective
of Section 3.

In Section 4 we discuss yet another extension of Young's inequality, based on
recent work done by J. Jak\v{s}eti\'{c} and J. E. Pe\v{c}ari\'{c} \cite{P}.

The paper ends by noticing the connection of our result to the theory of
$c$-convexity (that is, of convexity associated to a cost density function).

Last but not the least, all results in this paper can be extended verbatim to
the framework of nondecreasing functions $f:[a_{0},a_{1})\rightarrow\lbrack
A_{0},A_{1})$ such that $a_{0}<a_{1}\leq\infty$ and $A_{0}<A_{1}\leq\infty,$
$f(a_{0})=A_{0}$ and $\lim_{x\rightarrow a_{1}}f(x)=A_{1}.$ In other words,
the interval $[0,\infty)$ plays no special role in Young's inequality.

Besides, there is a straightforward companion of Young's inequality for
nonincreasing functions, but this is outside the scope of the present paper.

\section{Young's inequality for weighted measures}

In what follows $f:\left[  0,\infty\right)  \longrightarrow\left[
0,\infty\right)  $ will denote a nondecreasing function such that $f\left(
0\right)  =0$ and $\underset{x\rightarrow\infty}{\lim}f\left(  x\right)
=\infty.$ Since $f$ is not necessarily injective we will attach to $f$ a
\emph{pseudo-inverse} by the following formula:%
\[
f_{\sup}^{-1}:\left[  0,\infty\right)  \longrightarrow\left[  0,\infty\right)
,\quad f_{\sup}^{-1}\left(  y\right)  =\inf\{x\geq0:f(x)>y\}.
\]

Clearly, $f_{\sup}^{-1}$ is nondecreasing and $f_{\sup}^{-1}\left(  f\left(
x\right)  \right)  \geq x$ for all $x.$ Moreover, with the convention
$f(0-)=0,$
\[
f_{\sup}^{-1}\left(  y\right)  =\sup\left\{  x:y\in\left[  f\left(  x-\right)
,f\left(  x+\right)  \right]  \right\}  ;
\]
here $f\left(  x-\right)  $ and $f\left(  x+\right)  $ represent the lateral
limits at $x$. When $f$ is also continuous,%
\[
f_{\sup}^{-1}(y)=\max\left\{  x\geq0:y=f(x)\right\}  .
\]

\begin{remark}
$($F. Cunningham Jr. and N. Grossman \cite{CG1971}$)$. \emph{Since
pseudo-inverses will be used as integrands, it is convenient to enlarge the
concept of pseudo-inverse by referring to any function} $g$ \emph{such that}%
\[
f_{\inf}^{-1}\leq g\leq f_{\sup}^{-1},
\]
\emph{where} $f_{\inf}^{-1}(y)=\sup\{x\geq0:f(x)<y\}$. \emph{Necessarily,} $g$
\emph{is nondecreasing and any two} \emph{pseudo-inverses agree except on a
countable set (so their integrals will be the same)}.
\end{remark}

Given $0\leq a<b,$ we define the \emph{epigraph} and the \emph{hypograph} of
$f|_{[a,b]}$ respectively by%
\[
\operatorname{epi}f|_{[a,b]}=\left\{  \left(  x,y\right)  \in\left[
a,b\right]  \times\left[  f\left(  a\right)  ,f\left(  b\right)  \right]
:y\geq f\left(  x\right)  \right\}  ,
\]
and%
\[
\operatorname{hyp}f|_{[a,b]}=\left\{  \left(  x,y\right)  \in\left[
a,b\right]  \times\left[  f\left(  a\right)  ,f\left(  b\right)  \right]
:y\leq f\left(  x\right)  \right\}  .
\]

Their intersection is the \emph{graph} of $f|_{[a,b]},$
\[
\operatorname*{graph}f|_{[a,b]}=\left\{  \left(  x,y\right)  \in\left[
a,b\right]  \times\left[  f\left(  a\right)  ,f\left(  b\right)  \right]
:y=f\left(  x\right)  \right\}  .
\]

Notice that our definitions of epigraph and hypograph are not the standard
ones, but agree with them in the context of monotone functions.

We will next consider a measure $\rho$ on $\left[  0,\infty\right)
\times\left[  0,\infty\right)  ,$ which is locally absolutely continuous with
respect to the Lebesgue measure $dxdy,$ that is, $\rho$ is of the form
\[
\rho\left(  A\right)  =\int_{A}K\left(  x,y\right)  dxdy,
\]
where $K:\left[  0,\infty\right)  \times\left[  0,\infty\right)
\longrightarrow\lbrack0,\infty)\ $is a Lebesgue locally integrable function,
and $A$ is any compact subset of $\left[  0,\infty\right)  \times\left[
0,\infty\right)  $.

Clearly,%
\begin{align*}
\rho\left(  \operatorname{hyp}f|_{[a,b]}\right)  +\rho\left(
\operatorname{epi}f|_{[a,b]}\right)   &  =\rho\left(  \left[  a,b\right]
\times\left[  f\left(  a\right)  ,f\left(  b\right)  \right]  \right) \\
&  =\int_{a}^{b}\int_{f\left(  a\right)  }^{f\left(  b\right)  }K\left(
x,y\right)  dydx.
\end{align*}

Moreover,%
\[
\rho\left(  \operatorname{hyp}f|_{[a,b]}\right)  =\int_{a}^{b}\left(
\int_{f\left(  a\right)  }^{f\left(  x\right)  }K\left(  x,y\right)
dy\right)  dx.
\]
and%
\begin{equation}
\rho\left(  \operatorname{epi}f|_{[a,b]}\right)  =\int_{f\left(  a\right)
}^{f\left(  b\right)  }\left(  \int_{a}^{f_{\sup}^{-1}\left(  y\right)
}K\left(  x,y\right)  dx\right)  dy.\nonumber
\end{equation}

The discussion above can be summarized as follows:

\begin{lemma}
\label{Lem3}Let $f:\left[  0,\infty\right)  \longrightarrow\left[
0,\infty\right)  $ be a nondecreasing function such that $f\left(  0\right)
=0$ and $\underset{x\rightarrow\infty}{\lim}f\left(  x\right)  =\infty$. Then
for every Lebesgue locally integrable function $K:\left[  0,\infty\right)
\times\left[  0,\infty\right)  \longrightarrow\lbrack0,\infty)$ and every pair
of nonnegative numbers $a<b,$%
\begin{multline*}
\int_{a}^{b}\left(  \int_{f\left(  a\right)  }^{f\left(  x\right)  }K\left(
x,y\right)  dy\right)  dx+\int_{f\left(  a\right)  }^{f\left(  b\right)
}\left(  \int_{a}^{f_{\sup}^{-1}\left(  y\right)  }K\left(  x,y\right)
dx\right)  dy\\
=\int_{a}^{b}\int_{f\left(  a\right)  }^{f\left(  b\right)  }K\left(
x,y\right)  dydx.
\end{multline*}

\end{lemma}

We can now state the main result of this section:

\begin{theorem}
\label{ThmYoungNondecr}\emph{(}Young's inequality for nondecreasing
functions\textbf{\emph{)}. }Under the assumptions of Lemma $3$, for every pair
of nonnegative numbers $a<b,$ and every number $c\geq f(a)$ we have
\begin{multline*}
\int_{a}^{b}\int_{f\left(  a\right)  }^{c}K\left(  x,y\right)  dydx\\
\leq\int_{a}^{b}\left(  \int_{f\left(  a\right)  }^{f\left(  x\right)
}K\left(  x,y\right)  dy\right)  dx+\int_{f\left(  a\right)  }^{c}\left(
\int_{a}^{f_{\sup}^{-1}\left(  y\right)  }K\left(  x,y\right)  dx\right)  dy.
\end{multline*}
If in addition $K$ is strictly positive almost everywhere, then the equality
occurs if and only if $c\in\left[  f\left(  b-\right)  ,f\left(  b+\right)
\right]  .$

\begin{proof}
We start with the case where $f\left(  a\right)  \leq c\leq f\left(
b-\right)  $. See Figure \ref{fig1}. \
\begin{figure}
[h]
\begin{center}
\includegraphics[
height=5.7464cm,
width=6.491cm
]%
{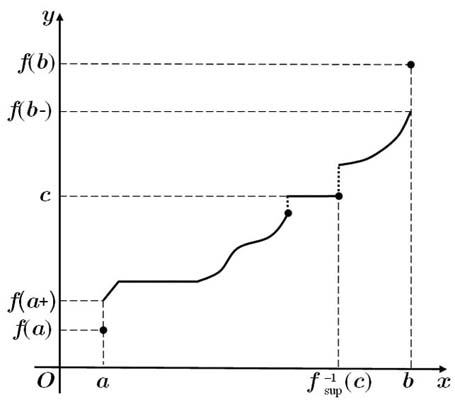}%
\caption{The geometry of Young's inequality when $f\left(  a\right)  \leq
c\leq f\left(  b-\right)  .$}%
\label{fig1}%
\end{center}
\end{figure}

In this case,%
\begin{multline*}
\int_{a}^{b}\left(  \int_{f\left(  a\right)  }^{f\left(  x\right)  }K\left(
x,y\right)  dy\right)  dx+\int_{f\left(  a\right)  }^{c}\left(  \int
_{a}^{f_{\sup}^{-1}\left(  y\right)  }K\left(  x,y\right)  dx\right)  dy\\
=\int_{a}^{f_{\sup}^{-1}\left(  c\right)  }\left(  \int_{f\left(  a\right)
}^{f\left(  x\right)  }K\left(  x,y\right)  dy\right)  dx+\int_{f\left(
a\right)  }^{c}\left(  \int_{a}^{f_{\sup}^{-1}\left(  y\right)  }K\left(
x,y\right)  dx\right)  dy\\
+\int_{f_{\sup}^{-1}\left(  c\right)  }^{b}\left(  \int_{f\left(  a\right)
}^{f\left(  x\right)  }K\left(  x,y\right)  dy\right)  dx\\
=\int_{a}^{f_{\sup}^{-1}\left(  c\right)  }\int_{f\left(  a\right)  }%
^{c}K\left(  x,y\right)  dydx+\int_{f_{\sup}^{-1}\left(  c\right)  }%
^{b}\left(  \int_{c}^{f\left(  x\right)  }K\left(  x,y\right)  dy\right)  dx\\
+\int_{f_{\sup}^{-1}\left(  c\right)  }^{b}\int_{f\left(  a\right)  }%
^{c}K\left(  x,y\right)  dydx\\
\geq\int_{a}^{b}\int_{f\left(  a\right)  }^{c}K\left(  x,y\right)  dydx,
\end{multline*}
with equality if and only if $\int_{f_{\sup}^{-1}\left(  c\right)  }%
^{b}\left(  \int_{c}^{f\left(  x\right)  }K\left(  x,y\right)  dy\right)
dx=0.$ When $K$ is strictly positive almost everywhere, this means that
$c=f\left(  b-\right)  $.

If $c\geq f\left(  b+\right)  ,$ then%
\begin{multline*}
\int_{a}^{b}\left(  \int_{f\left(  a\right)  }^{f\left(  x\right)  }K\left(
x,y\right)  dy\right)  dx+\int_{f\left(  a\right)  }^{c}\left(  \int
_{a}^{f_{\sup}^{-1}\left(  y\right)  }K\left(  x,y\right)  dx\right)  dy\\
=\int_{a}^{f_{\sup}^{-1}\left(  c\right)  }\left(  \int_{f\left(  a\right)
}^{f\left(  x\right)  }K\left(  x,y\right)  dy\right)  dx+\int_{f\left(
a\right)  }^{c}\left(  \int_{a}^{f_{\sup}^{-1}\left(  y\right)  }K\left(
x,y\right)  dx\right)  dy\\
-\int_{b}^{f_{\sup}^{-1}\left(  c\right)  }\left(  \int_{f\left(  a\right)
}^{f\left(  x\right)  }K\left(  x,y\right)  dy\right)  dx\\
=\int_{a}^{f_{\sup}^{-1}\left(  c\right)  }\int_{f\left(  a\right)  }%
^{c}K\left(  x,y\right)  dydx\\
-\left(  \int_{b}^{f_{\sup}^{-1}\left(  c\right)  }\left(  \int_{f\left(
a\right)  }^{f\left(  c\right)  }K\left(  x,y\right)  dy\right)
dx-\int_{f\left(  b+\right)  }^{c}\left(  \int_{a}^{f_{\sup}^{-1}\left(
y\right)  }K\left(  x,y\right)  dx\right)  dy\right)  \\
\geq\int_{a}^{b}\int_{f\left(  a\right)  }^{c}K\left(  x,y\right)  dydx.
\end{multline*}
Equality holds if and only if$\ \int_{f\left(  b+\right)  }^{c}\left(
\int_{a}^{f_{\sup}^{-1}\left(  y\right)  }K\left(  x,y\right)  dx\right)  dy$,
that is, when $c=f\left(  b+\right)  $ (provided that $K$ is strictly positive
almost everywhere). See Figure 2.%
\begin{figure}
[ptb]
\begin{center}
\includegraphics[
height=5.3927cm,
width=6.6558cm
]%
{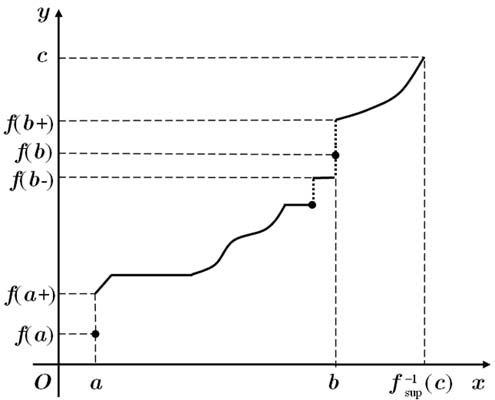}%
\caption{The case $c\geq f\left(  b+\right)  .$}%
\label{fig2}%
\end{center}
\end{figure}

If $c\in\left(  f\left(  b-\right)  ,f\left(  b+\right)  \right)  ,$ then
$f_{\sup}^{-1}\left(  c\right)  =b$ and the inequality in the statement of
Theorem \ref{ThmYoungNondecr} is actually an equality. See Figure \ref{fig3}.%
\begin{figure}
[h]
\begin{center}
\includegraphics[
height=6.1593cm,
width=6.9216cm
]%
{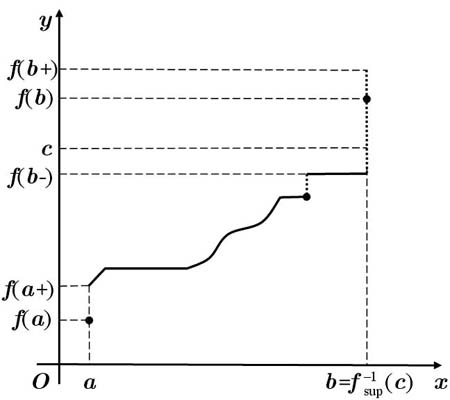}%
\caption{The equality case.}%
\label{fig3}%
\end{center}
\end{figure}

\end{proof}
\end{theorem}

\begin{corollary}
\label{CorContIncr}\emph{(}Young's inequality for continuous increasing
functions\emph{)}\textbf{. }If $f:\left[  0,\infty\right)  \longrightarrow
\left[  0,\infty\right)  $ is also continuous and increasing, then%
\begin{multline*}
\int_{a}^{b}\int_{f\left(  a\right)  }^{c}K\left(  x,y\right)  dydx\\
\leq\int_{a}^{b}\left(  \int_{f\left(  a\right)  }^{f\left(  x\right)
}K\left(  x,y\right)  dy\right)  dx+\int_{f\left(  a\right)  }^{c}\left(
\int_{a}^{f^{-1}\left(  y\right)  }K\left(  x,y\right)  dx\right)  dy
\end{multline*}
for every real number $c\geq f(a)$. Assuming $K$ strictly positive almost
everywhere, the equality occurs if and only if $\ c=f\left(  b\right)  .$
\end{corollary}

If $K\left(  x,y\right)  =1$ for every $x,y\in\left[  0,\infty\right)  $, then
Corollary \ref{CorContIncr} asserts that%
\[
bc-af\left(  a\right)  <\int_{a}^{b}f\left(  x\right)  dx+\int_{f\left(
a\right)  }^{c}f^{-1}\left(  y\right)  dy\text{\quad for all }0<a<b\text{ and
}c>f(a);
\]
equality occurs if and only if $c=f\left(  b\right)  $. In the special case
where $a=f\left(  a\right)  =0$, this reduces to the classical inequality of Young.

\begin{remark}
$($The probabilistic companion of Theorem \ref{ThmYoungNondecr}$)$.
\emph{Suppose there is given a nonnegative random variable} $X:[0,\infty
)\rightarrow\lbrack0,\infty)$ \emph{whose cumulative distribution function}
$F_{X}(x)=P\left(  X\leq x\right)  $ \emph{admits a density, that is, a
nonnegative Lebesgue-integrable function} $\rho_{X}$ \emph{such} \emph{that}
\[
P\left(  x\leq X\leq y\right)  =\int_{x}^{y}\rho_{X}(u)du\text{\quad for all
}x\leq y.
\]
\emph{The} quantile function \emph{of the distribution function} $F_{X}$
$($\emph{also known as the} increasing rearrangement \emph{of the random
variable} $X)$ \emph{is defined by}%
\[
Q_{X}(x)=\inf\left\{  y:F_{X}(y)\geq x\right\}  .
\]
\emph{Thus, a quantile function is nothing but a pseudo-inverse of} $F_{X}$.
\emph{Motivated by Statistics, a number of fast algorithms were developed for
computing the quantile functions with high accuracy. See} \cite{A}.
\emph{Without entering the details, we recall here the remarkable formula}
\emph{(due to} \emph{G.} \emph{Steinbrecher)} \emph{for the quantile function
of the normal distribution:}%
\[
\operatorname{erf}^{-1}(z)=\sum_{k=0}^{\infty}\frac{c_{k}\left(  \frac
{\sqrt{\pi}}{2}z\right)  ^{2k+1}}{2k+1},
\]
\emph{where} $c_{0}=1$ \emph{and}
\[
c_{k}=\sum_{m=0}^{k-1}\frac{c_{m}c_{k-m-1}}{\left(  m+1\right)  \left(
2m+1\right)  }\text{\quad\emph{for all} }k\geq1.
\]

\emph{According to Theorem} \ref{ThmYoungNondecr}, \emph{for every pair of
continuous random variables} $Y,Z:[0,\infty)\rightarrow\lbrack0,\infty)$
\emph{with density} $\rho_{Y,Z},$ \emph{and every positive numbers} $b$
\emph{and} $c,$ \emph{the following inequality holds:}%
\[
P\left(  Y\leq b;Z\leq c\right)  \leq\int_{0}^{b}\left(  \int_{0}^{F_{X}%
(x)}\rho_{Y,Z}\left(  x,y\right)  dy\right)  dx+\int_{0}^{c}\left(  \int
_{0}^{Q_{X}(y)}\rho_{Y,Z}\left(  x,y\right)  dx\right)  dy.
\]
\emph{This can be seen as a principle of uncertainty, since it shows that the
functions}
\[
x\rightarrow\int_{0}^{F_{X}(x)}\rho_{Y,Z}\left(  x,y\right)  dy\text{ and
}y\rightarrow\int_{0}^{Q_{X}(y)}\rho_{Y,Z}\left(  x,y\right)  dx
\]
\emph{cannot be made simultaneously small.}
\end{remark}

\begin{remark}
$($The higher dimensional analogue of Theorem $1).$ \emph{Consider a locally
absolutely continuous} \emph{kernel} $K:\left[  0,\infty\right)
\times...\times\left[  0,\infty\right)  \longrightarrow\lbrack0,\infty
),\ K=K\left(  s_{1},s_{2},...,s_{n}\right)  ,$ \emph{and a family} $\phi
_{1},...,\phi_{n}:[a_{i},b_{i}]\rightarrow\mathbb{R}\ $\emph{of nondecreasing
functions defined on subintervals of} $\left[  0,\infty\right)  .$
\emph{Then}
\begin{multline*}
\int_{\phi_{1}\left(  a_{1}\right)  }^{\phi_{1}\left(  b_{1}\right)  }%
\int_{\phi_{2}\left(  a_{2}\right)  }^{\phi_{2}\left(  b_{2}\right)  }%
\cdots\int_{\phi_{n}\left(  a_{n}\right)  }^{\phi_{n}\left(  b_{n}\right)
}K\left(  s_{1},s_{2},...,s_{n}\right)  ds_{n}...ds_{2}ds_{1}\\
\leq%
{\displaystyle\sum\limits_{i=1}^{n}}
\int_{\phi_{i}\left(  a_{i}\right)  }^{\phi_{i}\left(  b_{i}\right)  }\left(
\int_{\phi_{1}\left(  a_{1}\right)  }^{\phi_{1}\left(  s\right)  }\cdots
\int_{\phi_{n}\left(  a_{n}\right)  }^{\phi_{n}\left(  s\right)  }K\left(
s_{1},...,s_{n}\right)  ds_{n}...ds_{i+1}ds_{i-1}...ds_{1}\right)  ds.\
\end{multline*}

\emph{The proof is based on mathematical induction (which is left to the
reader). The above inequality cover the n-variable generalization of Young's
inequality as obtained by Oppenheim \cite{O1927} (as well as the main result
in \cite{Pa1992}).}
\end{remark}

The following stronger version of Corollary \ref{CorContIncr} incorporates the
Legendre duality.

\begin{theorem}
\label{extYoung}Let $f:\left[  0,\infty\right)  \longrightarrow\left[
0,\infty\right)  $ be a continuous nondecreasing function and $\Phi
:[0,\infty)\rightarrow\mathbb{R}$ a convex function whose conjugate is also
defined on $[0,\infty)$. Then for all $b>a\geq0,$ $c\geq f(a),$ and
$\varepsilon>0$ we have
\begin{multline*}
\int_{a}^{b}\Phi\left(  \varepsilon\int_{f\left(  a\right)  }^{f\left(
x\right)  }K\left(  x,y\right)  dy\right)  dx+\int_{f(a)}^{c}\Phi^{\ast
}\left(  \frac{1}{\varepsilon}\int_{a}^{f_{\sup}^{-1}\left(  y\right)
}K\left(  x,y\right)  dx\right)  dx\\
\geq\int_{a}^{b}\int_{f\left(  a\right)  }^{c}K\left(  x,y\right)
dydx-(c-f(a))\Phi\left(  \varepsilon\right)  -(b-a)\Phi^{\ast}\left(
1/\varepsilon\right)  .
\end{multline*}

\end{theorem}

\begin{proof}
According to the Legendre duality,%
\begin{equation}
\Phi(\varepsilon u)+\Phi^{\ast}(v/\varepsilon)\geq uv\text{\quad for all
}u,v,\varepsilon\geq0. \label{fi}%
\end{equation}
For $u=\int_{f\left(  a\right)  }^{f\left(  x\right)  }K\left(  x,y\right)
dy$ and $v=1$ we get%
\[
\Phi\left(  \varepsilon\int_{f\left(  a\right)  }^{f\left(  x\right)
}K\left(  x,y\right)  dy\right)  +\Phi^{\ast}\left(  1/\varepsilon\right)
\geq\int_{f\left(  a\right)  }^{f\left(  x\right)  }K\left(  x,y\right)  dy,
\]
and by integrating both sides from $a$ to $b$ we obtain the inequality%
\[
\int_{a}^{b}\Phi\left(  \varepsilon\int_{f\left(  a\right)  }^{f\left(
x\right)  }K\left(  x,y\right)  dy\right)  dx+(b-a)\Phi^{\ast}\left(
1/\varepsilon\right)  \geq\int_{a}^{b}\left(  \int_{f\left(  a\right)
}^{f\left(  x\right)  }K\left(  x,y\right)  dy\right)  dx.
\]
In a similar manner, starting with $u=1$ and $v=\int_{a}^{f_{\sup}^{-1}\left(
y\right)  }K\left(  x,y\right)  dx,$ we arrive first at the inequality
\[
\Phi\left(  \varepsilon\right)  +\Phi^{\ast}\left(  \frac{1}{\varepsilon}%
\int_{a}^{f_{\sup}^{-1}\left(  y\right)  }K\left(  x,y\right)  dx\right)
\geq\int_{a}^{f_{\sup}^{-1}\left(  y\right)  }K\left(  x,y\right)  dx,
\]
and then to%
\begin{multline*}
(c-f(a))\Phi\left(  \varepsilon\right)  +\int_{f(a)}^{c}\Phi^{\ast}\left(
\frac{1}{\varepsilon}\int_{a}^{f_{\sup}^{-1}\left(  y\right)  }K\left(
x,y\right)  dx\right)  dx\\
\geq\int_{f\left(  a\right)  }^{c}\left(  \int_{a}^{f_{\sup}^{-1}\left(
y\right)  }K\left(  x,y\right)  dx\right)  dy.
\end{multline*}
Therefore,%
\begin{multline*}
\int_{a}^{b}\Phi\left(  \varepsilon\int_{f\left(  a\right)  }^{f\left(
x\right)  }K\left(  x,y\right)  dy\right)  dx+\int_{f(a)}^{c}\Phi^{\ast
}\left(  \frac{1}{\varepsilon}\int_{a}^{f_{\sup}^{-1}\left(  y\right)
}K\left(  x,y\right)  dx\right)  dx\\
\geq\int_{a}^{b}\left(  \int_{f\left(  a\right)  }^{f\left(  x\right)
}K\left(  x,y\right)  dy\right)  dx+\int_{f\left(  a\right)  }^{c}\left(
\int_{a}^{f_{\sup}^{-1}\left(  y\right)  }K\left(  x,y\right)  dx\right)  dy\\
-(b-a)\Phi^{\ast}\left(  1/\varepsilon\right)  -(c-f(a))\Phi\left(
\varepsilon\right)  .
\end{multline*}
According to Theorem \ref{ThmYoungNondecr},%
\begin{multline*}
\int_{a}^{b}\left(  \int_{f\left(  a\right)  }^{f\left(  x\right)  }K\left(
x,y\right)  dy\right)  dx+\int_{f\left(  a\right)  }^{c}\left(  \int
_{a}^{f_{\sup}^{-1}\left(  y\right)  }K\left(  x,y\right)  dx\right)  dy\\
\geq\int_{a}^{b}\int_{f\left(  a\right)  }^{c}K\left(  x,y\right)  dydx,
\end{multline*}
and the inequality in the statement of Theorem \ref{extYoung} is now clear.
\end{proof}

In the special case where $K\left(  x,y\right)  =1,$ $a=f\left(  a\right)  =0$
and $\Phi(x)=x^{p}/p$ (for some\emph{\ }$p>1$), Theorem \ref{extYoung} yields
the following inequality:%
\[
\int_{0}^{b}f^{p}\left(  x\right)  dx+\int_{0}^{c}\left(  f_{\sup}^{-1}\left(
y\right)  \right)  ^{p}dy\geq pbc-\left(  p-1\right)  \left(  b+c\right)
,\ \text{for every }b,c\geq0.
\]

This remark extends a result due to W. T. Sulaiman \cite{S}.

We end this section by noticing the following result that complements Theorem
\ref{ThmYoungNondecr}.

\begin{proposition}
\label{PropMerkle}Under the assumptions of Lemma \ref{Lem3},%
\begin{multline*}
\int_{a}^{b}\left(  \int_{f\left(  a\right)  }^{f\left(  x\right)  }K\left(
x,y\right)  dy\right)  dx+\int_{f\left(  a\right)  }^{c}\left(  \int
_{a}^{f_{\sup}^{-1}\left(  y\right)  }K\left(  x,y\right)  dx\right)  dy\\
\leq\max\left\{  \int_{a}^{b}\int_{f\left(  a\right)  }^{f\left(  b\right)
}K\left(  x,y\right)  dydx,\int_{a}^{f_{\sup}^{-1}\left(  c\right)  }%
\int_{f\left(  a\right)  }^{c}K\left(  x,y\right)  dydx\right\}  .
\end{multline*}
Assuming $K$ strictly positive almost everywhere, the equality occurs if and
only if $c=f\left(  b\right)  .$
\end{proposition}

\begin{proof}
If $c<f\left(  b\right)  $, then from Lemma \ref{Lem3} we infer that%
\begin{multline*}
\int_{a}^{b}\left(  \int_{f\left(  a\right)  }^{f\left(  x\right)  }K\left(
x,y\right)  dy\right)  dx+\int_{f\left(  a\right)  }^{c}\left(  \int
_{a}^{f_{\sup}^{-1}\left(  y\right)  }K\left(  x,y\right)  dx\right)  dy\\
=\int_{a}^{b}\left(  \int_{f\left(  a\right)  }^{f\left(  x\right)  }K\left(
x,y\right)  dy\right)  dx+\int_{f\left(  a\right)  }^{f\left(  b\right)
}\left(  \int_{a}^{f_{\sup}^{-1}\left(  y\right)  }K\left(  x,y\right)
dx\right)  dy\\
-\int_{c}^{f\left(  b\right)  }\left(  \int_{a}^{f_{\sup}^{-1}\left(
y\right)  }K\left(  x,y\right)  dx\right)  dy\\
\leq\int_{a}^{b}\int_{f\left(  a\right)  }^{f\left(  b\right)  }K\left(
x,y\right)  dydx
\end{multline*}
The other case, $c\geq f\left(  b\right)  $, has a similar approach.
\end{proof}

Proposition \ref{PropMerkle} extends a result due to M. J. Merkle \cite{Me}.

\section{The precision in Young's inequality}

The main result of this section is as follows:

\begin{theorem}
\label{ThmPrec}Under the assumptions of Lemma \ref{Lem3}, for all $b\geq
a\geq0$ and $c\geq f(a),$%
\begin{multline*}
\int_{a}^{b}\left(  \int_{f\left(  a\right)  }^{f\left(  x\right)  }K\left(
x,y\right)  dy\right)  dx+\int_{f\left(  a\right)  }^{c}\left(  \int
_{a}^{f_{\sup}^{-1}\left(  y\right)  }K\left(  x,y\right)  dx\right)  dy\\
-\int_{a}^{b}\int_{f\left(  a\right)  }^{c}K\left(  x,y\right)  dydx\leq
\left\vert \int_{f_{\sup}^{-1}\left(  c\right)  }^{b}\int_{c}^{f\left(
b\right)  }K\left(  x,y\right)  dydx\right\vert \text{.}%
\end{multline*}
Assuming $K$ strictly positive almost everywhere, the equality occurs if and
only if $c=f\left(  b\right)  $.
\end{theorem}

\begin{proof}
The case where $f\left(  a\right)  \leq c\leq f\left(  b-\right)  $ is
illustrated in Figure \ref{fig4}. The left-hand side of the inequality in the
statement of Theorem \ref{ThmPrec} represents the measure of the cross-hatched
curvilinear trapezium, while right-hand side is the measure of the $ABCD$
rectangle.%
\begin{figure}
[h]
\begin{center}
\includegraphics[
height=5.4696cm,
width=6.9787cm
]%
{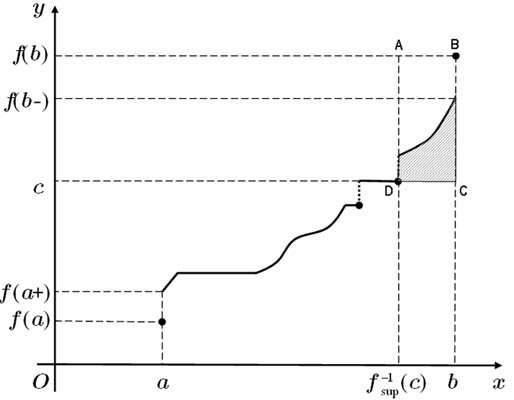}%
\caption{The geometry of the case $f\left(  a\right)  \leq c\leq f\left(
b-\right)  .$}%
\label{fig4}%
\end{center}
\end{figure}

Therefore,%
\begin{multline*}
\int_{a}^{b}\left(  \int_{f\left(  a\right)  }^{f\left(  x\right)  }K\left(
x,y\right)  dy\right)  dx+\int_{f\left(  a\right)  }^{c}\left(  \int
_{a}^{f_{\sup}^{-1}\left(  y\right)  }K\left(  x,y\right)  dx\right)  dy\\
-\int_{a}^{b}\int_{f\left(  a\right)  }^{c}K\left(  x,y\right)  dydx=\int
_{f_{\sup}^{-1}\left(  c\right)  }^{b}\left(  \int_{c}^{f\left(  x\right)
}K\left(  x,y\right)  dy\right)  dx\\
\leq\int_{f_{\sup}^{-1}\left(  c\right)  }^{b}\int_{c}^{f\left(  b\right)
}K\left(  x,y\right)  dydx.
\end{multline*}
The equality holds if and only if $\int_{f_{\sup}^{-1}\left(  c\right)  }%
^{b}\left(  \int_{c}^{f\left(  x\right)  }K\left(  x,y\right)  dy\right)
dx=0,$ that is, when $f\left(  b-\right)  =c.$

The case where $c\geq f\left(  b+\right)  $ is similar to the precedent one.
The first term will be:%
\begin{multline*}
\int_{a}^{b}\left(  \int_{f\left(  a\right)  }^{f\left(  x\right)  }K\left(
x,y\right)  dy\right)  dx+\int_{f\left(  a\right)  }^{c}\left(  \int
_{a}^{f_{\sup}^{-1}\left(  y\right)  }K\left(  x,y\right)  dx\right)  dy\\
-\int_{a}^{b}\int_{f\left(  a\right)  }^{c}K\left(  x,y\right)  dydx=\int
_{b}^{f_{\sup}^{-1}\left(  c\right)  }\left(  \int_{f\left(  b\right)
}^{f\left(  x\right)  }K\left(  x,y\right)  dy\right)  dx\\
\leq\int_{b}^{f_{\sup}^{-1}\left(  c\right)  }\int_{f\left(  b\right)  }%
^{c}K\left(  x,y\right)  dydx.
\end{multline*}
Equality holds if and only if $\int_{b}^{f_{\sup}^{-1}\left(  c\right)  }%
\int_{f\left(  b\right)  }^{c}K\left(  x,y\right)  dydx=0,\ $so we must have
$f\left(  b+\right)  =c$.

The case where $c\in\left[  f\left(  b-\right)  ,f\left(  b+\right)  \right]
$ is trivial, both sides of our inequality being equal to zero.
\end{proof}

\begin{corollary}
\label{CorMing}\emph{(}E. Minguzzi\emph{\ }\cite{M}\emph{)}. If moreover
$K\left(  x,y\right)  =1$ on $\left[  0,\infty\right)  \times\left[
0,\infty\right)  $, and $f$ is continuous and increasing, then
\[
\int_{a}^{b}f\left(  x\right)  dx+\int_{f\left(  a\right)  }^{c}f^{-1}\left(
y\right)  dy\ -bc+af\left(  a\right)  \leq\left(  f^{-1}\left(  c\right)
-b\right)  \cdot\left(  c-f\left(  b\right)  \right)  .
\]
The equality occurs if and only if $c=f\left(  b\right)  $.
\end{corollary}

More accurate bounds can be indicated under the presence of convexity.

\begin{corollary}
\label{CorJP}Let $f$ be a nondecreasing continuous function, which is convex
on the interval $\left[  \min\left\{  f_{\sup}^{-1}\left(  c\right)
,b\right\}  ,\max\left\{  f_{\sup}^{-1}\left(  c\right)  ,b\right\}  \right]
$. Then:%
\begin{multline*}
i)~\int_{a}^{b}\left(  \int_{f\left(  a\right)  }^{f\left(  x\right)
}K\left(  x,y\right)  dy\right)  dx+\int_{f\left(  a\right)  }^{c}\left(
\int_{a}^{f_{\sup}^{-1}\left(  y\right)  }K\left(  x,y\right)  dx\right)  dy\\
-\int_{a}^{b}\int_{f\left(  a\right)  }^{c}K\left(  x,y\right)  dydx\\
\leq\int_{f_{\sup}^{-1}\left(  c\right)  }^{b}\int_{c}^{c+\frac{f(b)-c}%
{b-f_{\sup}^{-1}\left(  c\right)  }(x-f_{\sup}^{-1}\left(  c\right)
)}K\left(  x,y\right)  dydx\text{,\quad for every }c\leq f\left(  b\right)  ;
\end{multline*}%
\begin{multline*}
ii)~\int_{a}^{b}\left(  \int_{f\left(  a\right)  }^{f\left(  x\right)
}K\left(  x,y\right)  dy\right)  dx+\int_{f\left(  a\right)  }^{c}\left(
\int_{a}^{f_{\sup}^{-1}\left(  y\right)  }K\left(  x,y\right)  dx\right)  dy\\
-\int_{a}^{b}\int_{f\left(  a\right)  }^{c}K\left(  x,y\right)  dydx\\
\geq\int_{b}^{f_{\sup}^{-1}\left(  c\right)  }\int_{f\left(  b\right)
}^{f(b)+\frac{c-f(b)}{f_{\sup}^{-1}\left(  c\right)  -b}(x-b)}K\left(
x,y\right)  dydx\text{,\quad for every }c\geq f\left(  b\right)  .
\end{multline*}
If $\ f$ is concave on the aforementioned interval, then the inequalities
above work in the reverse way.

Assuming $K$ strictly positive almost everywhere, the equality occurs if and
only if$\ f$ is an affine function or $f\left(  b\right)  =c$.

\begin{proof}
We will restrict here to the case of convex functions, the argument for the
concave functions being similar.

The left-hand side term of each of the inequalities in our statement
represents the measure of the cross-hatched surface. See Figure 5 and Figure 6.

\medskip%
\raisebox{-0cm}{\parbox[b]{5.9638cm}{\begin{center}
\includegraphics[
height=4.9446cm,
width=5.9638cm
]%
{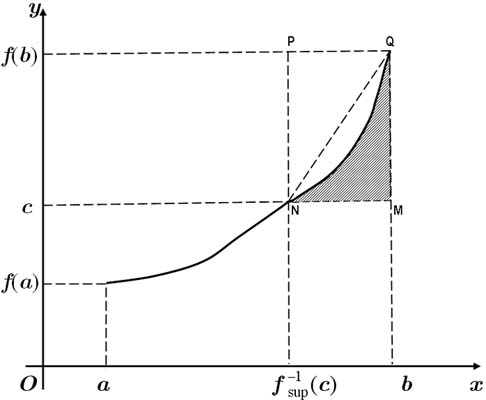}%
\\
Figure 5. The geometry of the case $c\leq f\left(  b\right)  .$%
\end{center}}}%
\raisebox{-0cm}{\parbox[b]{5.9594cm}{\begin{center}
\includegraphics[
height=4.9776cm,
width=5.9594cm
]%
{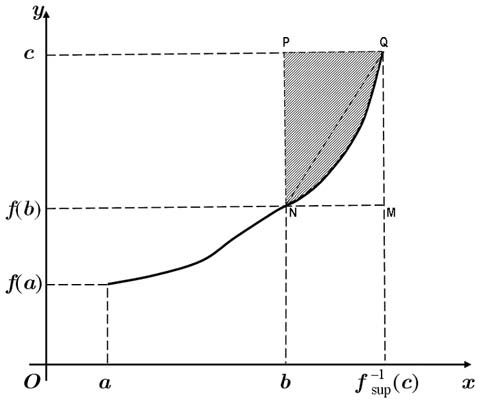}%
\\
Figure 6. The geometry of the case $c\geq f\left(  b\right)  .$%
\end{center}}}%
\qquad\medskip

As the points of the graph of the convex function $f$ (restricted to the
interval of endpoints $b$ and $f_{\sup}^{-1}\left(  c\right)  )$ are under the
chord joining $\left(  b,f\left(  b\right)  \right)  $ and $\left(  f_{\sup
}^{-1}\left(  c\right)  ,c\right)  ,$ it follows that this measure is less
than the measure of the enveloping triangle $MNQ$ when $c\leq f(b).$ This
yields $i)$. The assertion $ii)$ follows in a similar way.
\end{proof}
\end{corollary}

Corollary \ref{CorJP} extends a result due to J. Jak\v{s}eti\'{c} and J. E.
Pe\v{c}ari\'{c}\textit{\ \cite{P}.} They considered the special case were
$K\left(  x,y\right)  =1$ on $\left[  0,\infty\right)  \times\left[
0,\infty\right)  $ and $f:\left[  0,\infty\right)  \rightarrow\left[
0,\infty\right)  $ is increasing and differentiable, with an increasing
derivative on the interval $\left[  \min\left\{  f^{-1}\left(  c\right)
,b\right\}  ,\max\left\{  f^{-1}\left(  c\right)  ,b\right\}  \right]  $ and
$f(0)=0.$ In this case the conclusion of Corollary \ref{CorJP} reads as
follows:%
\begin{align*}
i)\text{ }\int_{0}^{b}f\left(  x\right)  dx+\int_{0}^{c}f^{-1}\left(
y\right)  dy\ -bc  &  \leq\frac{1}{2}\left(  f^{-1}\left(  c\right)
-b\right)  \left(  c-f\left(  b\right)  \right)  \ \text{for }c<f\left(
b\right)  ;\\
ii)\text{ }\int_{0}^{b}f\left(  x\right)  dx+\int_{0}^{c}f^{-1}\left(
y\right)  dy\ -bc  &  \geq\frac{1}{2}\left(  f^{-1}\left(  c\right)
-b\right)  \left(  c-f\left(  b\right)  \right)  \ \text{for }c>f\left(
b\right)  .
\end{align*}

The equality holds if $f\left(  b\right)  =c$ or $f$ is an affine function.
The inequality sign should be reversed if $f$ has a decreasing derivative on
the interval
\[
\left[  \min\left\{  f^{-1}\left(  c\right)  ,b\right\}  ,\max\left\{
f^{-1}\left(  c\right)  ,b\right\}  \right]  .
\]

\section{The connection with $c$-convexity}

Motivated by the mass transportation theory, several people \cite{D1988},
\cite{EN1974} drew a parallel to the classical theory of convex functions by
extending the Legendre duality. Technically, given two compact metric spaces
$X$ and $Y$ and a \emph{cost density} function $c:X\times Y\rightarrow
\mathbb{R}$ (which is supposed to be continuous), we may consider the
following generalization of the notion of convex function:

\begin{definition}
\label{cConv}A function $F:X\rightarrow\mathbb{R}$ is $c$-convex if there
exists a function $G:Y\rightarrow\mathbb{R}$ such that
\begin{equation}
F(x)=\sup_{y\in Y}\left\{  c(x,y)-G(y)\right\}  ,\;\text{for all }x\in X.
\label{c-conv}%
\end{equation}

\end{definition}

We abbreviate (\ref{c-conv}) by writing $F=G^{c}$. A useful remark is the
equality%
\[
F^{cc}=F,
\]
that is,%
\begin{equation}
F(x)=\sup_{y\in Y}\left\{  c(x,y)-F^{c}(y)\right\}  ,\;\text{for all }x\in X.
\label{cdual}%
\end{equation}

The classical notion of convex function corresponds to the case where $X$ is a
compact interval and $c(x,y)=xy$. The details can be found in \cite{NP2006},
pp. 40-42.

Theorem \ref{ThmYoungNondecr} illustrates the theory of $c$-convex functions
for the spaces $X=[a,\infty]$, $Y=[f(a),\infty]$ (the Alexandrov one point
compactification of $[a,\infty)$ and respectively $[f(a),\infty)$), and the
cost function
\begin{equation}
c(x,y)=\int_{a}^{x}\int_{f(a)}^{y}K\left(  s,t\right)  dtds\text{.}
\label{cKrelation}%
\end{equation}
In fact, under the hypotheses of this theorem, the functions%
\[
F(x)=\int_{a}^{x}\left(  \int_{f\left(  a\right)  }^{f\left(  s\right)
}K\left(  s,t\right)  dt\right)  ds,\quad x\geq a,
\]
and%
\[
G(y)=\int_{f\left(  a\right)  }^{y}\left(  \int_{a}^{f_{\sup}^{-1}\left(
t\right)  }K\left(  s,t\right)  ds\right)  dt,\quad y\geq f(a),
\]
verify the relations $F^{c}=G$ and $G^{c}=F$ (due to the equality case as
specified in the statement of Theorem \ref{ThmYoungNondecr}, so they are both
$c$-convex.

On the other hand, a simple argument shows that $F$ and $G$ are also convex in
the usual sense.

Let us call the functions $c$ that admits a representation of the form
(\ref{cKrelation}) with $K\in L^{1}(\mathbb{R\times R}),$ \emph{absolutely}
\emph{continuous in the hyperbolic sense}. With this terminology, Theorem
\ref{ThmYoungNondecr} can be rephrased as follows:

\begin{theorem}
\label{ThmcConv}Suppose that $c:[a,b]\times\lbrack A,B]\rightarrow\mathbb{R}$
is an absolutely continuous function in the hyperbolic sense with mixed
derivative $\frac{\partial^{2}c}{\partial x\partial y}\geq0,$ and
$f:[a,b]\rightarrow\lbrack A,B]$ is a nondecreasing function such that
$f(a)=A.$ Then%
\begin{equation}
c(x,y)-c(a,f(a))\leq\int_{a}^{x}\frac{\partial c}{\partial t}(t,f(t))dt+\int
_{f(a)}^{y}\frac{\partial c}{\partial s}(f_{\sup}^{-1}(s),s)ds \label{cyineq}%
\end{equation}
for all $(x,y)\in\lbrack a,A]\times\lbrack b,B]$.

If $\frac{\partial^{2}c}{\partial x\partial y}>0$ almost everywhere,\ then
(\ref{cyineq}) becomes an equality if and only if $y\in\left[
f(x-),f(x+)\right]  ;$ here we made the convention $f(a-)=f(a)$ and
$f(b+)=f(b).$
\end{theorem}

Necessarily, an absolutely continuous function $c$ in the hyperbolic sense, is
continuous. It admits partial derivatives of the first order and a mixed
derivative $\frac{\partial^{2}c}{\partial x\partial y}$ almost everywhere.
Besides, the functions $y\rightarrow\frac{\partial c}{\partial x}(x,y)$ and
$x\rightarrow\frac{\partial c}{\partial y}(x,y)$ are defined everywhere in
their interval of definition and represent absolutely continuous functions;
they are also nondecreasing provided that $\frac{\partial^{2}c}{\partial
x\partial y}\geq0$ almost everywhere.

A special case of Theorem \ref{ThmcConv} was proved by Zs. P\'{a}les
\cite{Pa1990}, \cite{Pa1992} (assuming $c:[a,A]\times\lbrack b,B]\rightarrow
\mathbb{R}$ a continuously differentiable function with nondecreasing
derivatives $y\rightarrow\frac{\partial c}{\partial x}(x,y)$ and
$x\rightarrow\frac{\partial c}{\partial y}(x,y),$ and $f:[a,b]\rightarrow
\lbrack A,B]$ an increasing homeomorphism). An example which escapes his
result but is covered by Theorem \ref{ThmcConv} is offered by the function%
\[
c(x,y)=\int_{0}^{x}\left\{  \frac{1}{s}\right\}  ds\int_{0}^{y}\left\{
\frac{1}{t}\right\}  dt,\,\quad x,y\geq0,
\]
where $\left\{  \frac{1}{s}\right\}  $ denotes the fractional part of
$\frac{1}{s}$ if $s>0,$ and $\left\{  \frac{1}{s}\right\}  =0$ if $s=0$.
According to Theorem \ref{ThmcConv},%
\begin{multline*}
\int_{0}^{x}\left\{  \frac{1}{s}\right\}  ds\int_{0}^{y}\left\{  \frac{1}%
{t}\right\}  dt\\
\leq\int_{0}^{x}\left(  \left\{  \frac{1}{s}\right\}  \int_{0}^{f(s)}\left\{
\frac{1}{t}\right\}  dt\right)  ds+\int_{0}^{y}\left(  \left\{  \frac{1}%
{t}\right\}  \int_{0}^{f_{\sup}^{-1}(t)}\left\{  \frac{1}{s}\right\}
ds\right)  dt,
\end{multline*}
for every nondecreasing function $f:[0,\infty)\rightarrow\lbrack0,\infty)$
such that $f(0)=0.$

\medskip

\noindent\textbf{Acknowledgement.} The authors were supported by CNCSIS Grant
PN2 ID\_$420.$

\end{document}